\documentclass[preprint,12pt,3p]{elsarticle}

\usepackage[utf8]{inputenc}

\usepackage{amssymb}
\usepackage{amsthm}
\usepackage{amsmath}
\usepackage{amsfonts}

 \newtheorem{thm}{Theorem}[section]
 \newtheorem{cor}[thm]{Corollary}
 \newtheorem{lem}[thm]{Lemma}
 \newtheorem{prop}[thm]{Proposition}
 \theoremstyle{definition}
 \newtheorem{defn}[thm]{Definition}
 \theoremstyle{remark}
 \newtheorem{rem}[thm]{Remark}
 
 \numberwithin{equation}{section}

\begin{document}

\begin{frontmatter}

\title{Fractional de la Vallée Poussin inequalities}

\author[label1]{Rui A. C. Ferreira}
\address[label1]{Grupo F\'isica-Matem\'atica, Faculdade de Ci\^encias, Universidade de Lisboa, Av. Prof. Gama Pinto, 2, 1649-003 Lisboa, Portugal}
%\address[label2]{Address Two\fnref{label4}}

%\cortext[cor1]{I am corresponding author}
%\fntext[label3]{I also want to inform about\ldots}
%\fntext[label4]{Small city}

\ead{raferreira@fc.ul.pt}
%\ead[url]{author-one-homepage.com}

\begin{abstract}
In this work we derive some inequalities for fractional boundary value problems, that generalize the well-known de la Vallée Poussin inequality. With our results we also were able to improve the intervals where some Mittag--Leffler functions don't possess real zeros.
\end{abstract}

\begin{keyword}
%% keywords here, in the form: keyword \sep keyword
de la Vallée Poussin inequality \sep Fractional derivatives \sep Lyapunov inequality.
 MSC codes here, in the form: \MSC 26D10  \sep 34A08. 
%% or \MSC[2008] code \sep code (2000 is the default)
\end{keyword}

\end{frontmatter}

%%
%% Start line numbering here if you want
%%
% \linenumbers

%% main text
\section{Introduction}

When considering a second order linear boundary value problem with Dirichlet boundary conditions, the following result is known as the de la Vallée Poussin inequality (see e.g. \cite{Hartman}):

\begin{thm}\label{VP}
Suppose that $x\in C^2[a,b]$ is a nontrivial solution of the BVP
\begin{align}
    x''+g(t)x'&+f(t)x=0,\quad t\in(a,b)\nonumber\\
    x(a)=&0=x(b),\label{bc0}
\end{align}
where $f,g\in C[a,b]$. Then, the following inequality holds:
\begin{equation}\label{V-P in}
    1<M_1(b-a)+M_2\frac{
(b-a)^2}{2},
\end{equation}
where $M_1=\max_{t\in[a,b]}|g(t)|$ and $M_2=\max_{t\in[a,b]}|f(t)|$.
\end{thm}
Cohn \cite{Cohn}, Harris \cite{Harris}, Hartman and Wintner \cite{Hartman}, and most recently the author \cite{Ferreira3} obtained generalizations of Theorem \ref{VP} in these referenced works, respectively. A survey about the de La Vallée Poussin work on boundary value problems maybe found in \cite{Mawhin}. The research in order to find de la Vallée Poussin or Lyapunov type inequalities is an endless subject (see e.g.  \cite{Mitrinovic}), but until 2013, it was done exclusively for classical ordinary differential equations. However, in that year the author presented for the first time in the literature \cite{Ferreira0} an inequality for a fractional differential equation depending on a fractional derivative. His result generalized the classical Lyapunov inequality (see \cite[Theorem 2.1]{Ferreira0}). Since then, many other researchers dedicated their time to find Lyapunov-type inequalities for boundary value problems in which fractional derivatives are present (see \cite{Agarwal,Cabrera1,Cabrera2,Jleli} and the references therein). It is, nevertheless, worth mentioning that there are some open problems within the subject \cite{Ferreira2}.

In this work we consider the fractional differential equation (see Section \ref{sec2} for a brief introduction to fractional calculus)
\begin{equation}\label{eq6969}
    (D_a^\alpha x)+g(t)(D_a^\beta x)+f(t)x=0,\quad 1<\alpha\leq 2,\ 0<\beta\leq 1,
\end{equation}
together with the boundary conditions \eqref{bc0}, and make an attempt to derive inequalities of de la Vallée Poussin type for such a problem. To the best of our knowledge it is the first time such results appear in the literature for an equation of the type given in \eqref{eq6969}. We divide our main results into two sections: in the first section we consider the differential equation $x''+g(t)(D_a^\beta x)+f(t)x=0$, while in the second one, we consider the differential equation $(D_a^\alpha x)+g(t)(D_a^\beta x)+f(t)x=0$. The main reason to do it so is that, when considering the first equation we were able to obtain results that generalize the ones by Hartman and Wintner \cite{Hartman} (and consequently of the de la Vallée Poussin), while when considering the second equation we were \emph{only} able to generalize the results of de la Vallée Poussin. Nevertheless, it is worth mentioning it that we obtain as a particular case from \eqref{eq6969}--\eqref{bc0}, i.e. considering $g=0$ on $[a,b]$, the Lyapunov fractional inequality \cite[Theorem 2.1]{Ferreira0}. Finally, we revisit some results (and provide some new ones) related with the zeros of certain Mittag--Leffler functions.

It  is  the  first  time  that  these  type  of  inequalities  appear  in  the  literature  for differential equations with a middle term (cf. \eqref{eq6969}) and as such we believe that this work might be a cornerstone for future research within this interesting subject.

\section{Fractional Calculus}\label{sec2}

We introduce here to the reader the basics about fractional integrals and derivatives, namely, what will be used throughout this work. A thorough introduction to the subject may be found in \cite{Kilbas}.

\begin{defn}
Let $\alpha\geq 0$ and $f$ be a real function defined on $[a,b]$.
The Riemann--Liouville fractional integral of order $\alpha$ is
defined by $(I^0_af)(x)=f(x)$ and
\begin{equation*}
(I_a^\alpha
f)(t)=\frac{1}{\Gamma(\alpha)}\int_a^t(t-s)^{\alpha-1}f(s)ds,\quad
\alpha>0,\quad t\in[a,b],
\end{equation*}
provided the integral exists.
\end{defn}

\begin{defn}
The Riemann--Liouville fractional derivative of order $n-1<\alpha\leq n$, $n\in\mathbb{N}$ of a function $f$ is defined by $(D_a^\alpha
f)(t)=(D^n I_a^{n-\alpha}f)(t)$, provided the right hand side of the equality exists. 
\end{defn}

The following result may be found in \cite[Property 2.2]{Kilbas}.

\begin{prop}\label{Prop}
Suppose that $f\in C[a,b]$ and let $q\geq p>0$. Then,
$$(D_a^p I_a^q f)(t)=(I_a^{q-p}f)(t),\quad t\in[a,b].$$
\end{prop}
A version of the mean value theorem is contained in the following
\begin{thm}\cite[Theorem 3.1]{Trujillo}\label{thm23}
Let $0<\beta\leq 1$. Suppose that $f\in C[a,t]$ is such that $(D_a^\beta f)\in C[a,t]$. Let $f(a)=0$. Then, there exists $\tau\in(a,t)$ such that 
\begin{equation}\label{ui}
f(t)=\frac{(t-a)^\beta}{\Gamma(\beta+1)}( D_a^\beta f)(\tau).
\end{equation}
\end{thm}

\section{Main results}

\subsection{The equation $x''+g(t)(D_a^\beta x)+f(t)x=0$}\label{sect1}

In this section we shall consider the following boundary value problem:
\begin{align}
    x''+g(t)(D_a^\beta x)&+f(t)x=0,\quad t\in(a,b),\ \beta\in(0,1],\label{eq0}\\
    x(a)=&0=x(b),\label{eq1}
\end{align}
where $f,g\in C[a,b]$. It follows the main result of this section:
\begin{thm}\label{thm0}
Suppose that $x\in C^2[a,b]$ is a solution of \eqref{eq0}--\eqref{eq1} such that $x(t)\neq 0$ for $t\in(a,b)$. Then, the following inequality holds:
\begin{multline}\label{in0}
    b-a<\max\left\{\int_a^{b} \frac{(s-a)^{2-\beta}}{\Gamma(2-\beta)}|g(s)|ds,\int_a^{b} \frac{(s-a)^{1-\beta}}{\Gamma(2-\beta)}(b-s)|g(s)|ds\right\}\\
    +\int_a^{b}(s-a)(b-s)|f(s)|ds.
\end{multline}
\end{thm}

\begin{proof}
We start by writing the BVP \eqref{eq0}--\eqref{eq1} in an equivalent integral form. Indeed, we know that $x\in C^2[a,b]$ is a solution of \eqref{eq0} if and only if it is a solution of 
$$x(t)=c_1+c_2(t-a)-\int_a^t(t-s)[g(s)(D_a^\beta x)(s)+f(s)x(s)]ds,$$
with $c_1,c_2\in\mathbb{R}$.

Now, since $x(a)=0$, then $c_1=0$. Also, since $x(b)=0$, then 
$$c_2=\frac{1}{b-a}\int_a^b(b-s)[g(s)(D_a^\beta x)(s)+f(s)x(s)] ds.$$
Therefore,
\begin{multline*}
    x(t)=\int_a^t\left[\frac{t-a}{b-a}(b-s)-(t-s)\right][g(s)(D_a^\beta x)(s)+f(s)x(s)]ds\\
    +\int_t^b\frac{t-a}{b-a}(b-s)[g(s)(D_a^\beta x)(s)+f(s)x(s)]ds,
\end{multline*}
which after some simplifications finally yields
\begin{multline*}
    (b-a)x(t)=\int_a^t(b-t)(s-a)[g(s)(D_a^\beta x)(s)+f(s)x(s)]ds\\
    +\int_t^b(t-a)(b-s)[g(s)(D_a^\beta x)(s)+f(s)x(s)] ds.
\end{multline*}
Differentiating both sides of the previous equality gives
\begin{equation}
    (b-a)x'(t)=-\int_a^t(s-a)[g(s)(D_a^\beta x)(s)+f(s)x(s)]ds+\int_t^b(b-s)[g(s)(D_a^\beta x)(s)+f(s)x(s)] ds.
\end{equation}
Let $\nu=\max_{t\in[a,b]}|x'(t)|>0$. Then, by the mean value theorem and the fact that $x(a)=0=x(b)$, we know that
$$|x(t)|\leq\nu(t-a),$$
and
$$|x(t)|\leq\nu(b-t),$$
for $t\in[a,b]$.
Therefore, 
\begin{equation}\label{in11}
    |x(t)|\leq\nu\phi(t),
\end{equation}
where $\phi(t)=\min(t-a,b-t)$, and it is clear that the $\leq$ in \eqref{in11}  is  a  $<$  for  some $t\in(a,b)$. Moreover, in view of $x(a)=0$, we have that\footnote{Note that if $\beta=1$ we immediately see that $|x'(t)|\leq \nu$ for all $t\in[a,b]$.}
$$|(D_a^\beta x)(t)|=\left|\frac{1}{\Gamma(1-\beta)}\int_a^t(t-s)^{-\beta}x'(s)ds\right|\leq\frac{\nu}{\Gamma(2-\beta)}(t-a)^{1-\beta},\quad 0<\beta<1,$$
where again the inequality is strict for some $t\in(a,b)$.
Therefore,
\begin{multline*}
     (b-a)|x'(t)|<\nu\int_a^t(s-a)\left[|g(s)|\frac{(s-a)^{1-\beta}}{\Gamma(2-\beta)}+|f(s)|\phi(s)\right] ds\\
     +\nu\int_t^b(b-s)\left[|g(s)|\frac{(s-a)^{1-\beta}}{\Gamma(2-\beta)}+|f(s)|\phi(s)\right]ds.
\end{multline*}
Note that the definition of $\phi$ shows that $(s-a)\phi(s)$ and $(b-s)\phi(s)$ are majorized by $(s-a)(b-s)$ on $[a,b]$, hence
\begin{multline}\label{in12}
     (b-a)|x'(t)|<\nu\left(\int_a^t\frac{(s-a)^{2-\beta}}{\Gamma(2-\beta)}|g(s)|ds+\int_t^b\frac{(s-a)^{1-\beta}}{\Gamma(2-\beta)}(b-s)|g(s)|ds\right)\\
     +\nu\int_a^b(s-a)(b-s)|f(s)|ds.
\end{multline}
Now, we define $S(t)=\int_a^t\frac{(s-a)^{2-\beta}}{\Gamma(2-\beta)}|g(s)|ds+\int_t^b\frac{(s-a)^{1-\beta}}{\Gamma(2-\beta)}(b-s)|g(s)|ds$ for $t\in[a,b]$. Then,
$$S'(t)=\frac{(t-a)^{2-\beta}}{\Gamma(2-\beta)}|g(t)|-\frac{(t-a)^{1-\beta}}{\Gamma(2-\beta)}(b-t)|g(t)|=(2t-(a+b))\frac{(t-a)^{1-\beta}}{\Gamma(2-\beta)}|g(t)|,$$
which means that $\max_{t\in[a,b]}S(t)$ is obtained either at $t=a$ or at $t=b$. It follows from \eqref{in12} that
\begin{multline*}
    b-a<\max\left\{\int_a^{b} \frac{(s-a)^{2-\beta}}{\Gamma(2-\beta)}|g(s)|ds,\int_a^{b} \frac{(s-a)^{1-\beta}}{\Gamma(2-\beta)}(b-s)|g(s)|ds\right\}\\
    +\int_a^{b}(s-a)(b-s)|f(s)|ds,
\end{multline*}
which concludes the proof.
\end{proof}
If we let $\beta=1$ in the previous theorem, then we immediately get Hartman and Wintner's result \cite{Hartman}:
\begin{cor}\label{corHartman}
Suppose that $x\in C^2[a,b]$ is a solution of 
\begin{align*}
    x''+g(t)x'&+f(t)x=0,\quad t\in(a,b),\\
    x(a)=&0=x(b),
\end{align*}
such that $x(t)\neq 0$ for $t\in(a,b)$. Then, the following inequality holds:
\begin{equation*}
    b-a<\max\left\{\int_a^{b} (s-a)|g(s)|ds,\int_a^{b} (b-s)|g(s)|ds\right\}
    +\int_a^{b}(s-a)(b-s)|f(s)|ds.
\end{equation*}
\end{cor}
\begin{rem}\label{rem90}
We note that if we assume in Theorem \ref{thm0} $x$ to be only nontrivial, then we may derive the inequality \eqref{in0} but with non-strict sign.
\end{rem}

We will end this section showing that, for certain values of the parameter $\beta$, we can improve a result obtained in \cite{Ferreira1}. For the sake of completeness we recall it now:
\begin{thm}\label{thm69}
Let $1<\alpha\leq 2$. Then, the Mittag--Leffler function 
$$E_{\alpha,2}(x)=\sum_{k=0}^\infty\frac{x^k}{\Gamma(k\alpha+2)},\quad x\in\mathbb{C},$$
has no real zeros for $$x\in\left[-\Gamma(\alpha)\frac{\alpha^\alpha}{(\alpha-1)^{\alpha-1}},0\right).$$
\end{thm}
In order to complete our goal, we first need the following
\begin{lem}\label{lem000}
Define the function
$$f(x)=\frac{x^x}{(x-1)^{x-1}},\quad x\in(1,2].$$
There exists a unique $x^\star\in(1,2)$ such that
$$f(x)<x+1,\ \forall x\in(1,x^\star),\ \mbox{and}\ f(x)>x+1,\ \forall x\in(x^\star,2].$$
\end{lem}
\begin{proof}
The function $g(x)=x+1$ is a straight line with $g(1)=2$ and $g(2)=3$. Now we show that $f$ is an increasing and concave function, with $\lim_{x\rightarrow 1}f(x)=1$ and $f(2)=4$, which in turn proves the result.

First, note that $x^x=e^{x\ln(x)}$. Therefore, $\lim_{x\rightarrow 0}x^x=1$, hence $\lim_{x\rightarrow 1}f(x)=1$.
Now, standard calculations show that
$$f'(x)=\frac{x^x}{(x-1)^{x-1}}(\ln(x)-\ln(x-1)).$$
Since $x/(x-1)>1$, then $f'>0$ and that shows that $f$ is increasing. Differentiating again and performing some simplifications, we obtain
$$f''(x)=\frac{x^{x-1}}{(x-1)^{x-1}}\left(x(\ln(x)-\ln(x-1))^2-\frac{1}{x-1}\right).$$
Defining the auxiliary function
$$h(x)=x(\ln(x)-\ln(x-1))^2-\frac{1}{x-1},$$
and differentiating it, we see that
$$h'(x)=\frac{((1-x)\ln(x-1)-1+(x-1)\ln(x))^2}{(x-1)^2}>0,\quad x\in(1,2].$$
Since $h(2)<0$ we conclude that $h(x)<0$ on $(1,2]$, i.e. $f''<0$ or, in other words, $f$ is concave on $(1,2]$. The proof is done.
\end{proof}
\begin{rem}
A numerical approximation of $x^\star$ of the previous lemma is given\footnote{This value was calculated using Maple Software} by 1.447.
\end{rem}
The following result improves Theorem \ref{thm69} in the sense that, for certain values of the parameter $\alpha$, the given Mittag--Leffler function cannot have zeros on a larger interval of real numbers.

\begin{thm}\label{thmzeros}
Let $1<\alpha<\overline{\alpha}$, where $\overline{\alpha}\in(1,2)$ is defined implicitly by $\frac{\overline{\alpha}^{\overline{\alpha}}}{(\overline{\alpha}-1)^{\overline{\alpha}-1}}=\overline{\alpha}+1$. Then, the Mittag--Leffler function $E_{\alpha,2}(x)$
has no real zeros for $$x\in\left(-\Gamma(\alpha)(1+\alpha),0\right)\supset\left[-\Gamma(\alpha)\frac{\alpha^\alpha}{(\alpha-1)^{\alpha-1}},0\right).$$
\end{thm}

\begin{proof}
By Lemma \ref{lem000}, the number $\overline{\alpha}$ is well defined.

Consider $a=0$ and $b=1$. Let $f=0$ in \eqref{eq0} and suppose that $x$ is a nontrivial solution of the following BVP
\begin{align*}
    x''(t)+\lambda(D_0^\beta x)(t)&=0,\quad t\in(0,1),\ \beta\in(0,1),\ \lambda\in\mathbb{R},\\
    x(0)=&0=x(1).
\end{align*}
By \cite[Corollary 5.3]{Kilbas} we may conclude that $\lambda$ must satisfy $E_{2-\beta,2}(-\lambda)=0$. It is clear that, if such $\lambda$ exist, it must be positive. By Theorem \ref{thm0} and Remark \ref{rem90}, we get that
$$1\leq\lambda\max\left\{\int_0^{1} \frac{s^{2-\beta}}{\Gamma(2-\beta)}ds,\int_0^{1} \frac{s^{1-\beta}}{\Gamma(2-\beta)}(1-s)ds\right\}=\frac{\lambda}{\Gamma(2-\beta)}\frac{1}{3-\beta}.$$
Therefore, putting $\alpha=2-\beta$ we conclude that if $x\in(-\Gamma(\alpha)(1+\alpha),0)$, then $E_{\alpha,2}(x)$ cannot have zeros. Since $\alpha<\overline{\alpha}$ we know, by Lemma \eqref{lem000}, that 
$$\frac{\alpha^\alpha}{(\alpha-1)^{\alpha-1}}<\alpha+1,$$
which concludes the proof.
\end{proof}

\subsection{The equation $(D_a^\alpha x)+g(t)(D_a^\beta x)+f(t)x=0$}

In this section we shall consider the following boundary value problem:
\begin{align}
    (D_a^\alpha x)+g(t)(D_a^\beta x)&+f(t)x=0,\quad t\in(a,b),\ \beta\in(0,1],\ \alpha\in(1, 2],\label{eq00}\\
    x(a)=&0=x(b),\label{eq11}
\end{align}
where $f,g\in C[a,b]$ and $\alpha-\beta-1\geq 0$. This BVP brings many differences in its study when compared to the one described in Section \ref{sect1}. For example, now, we don't even expect to have continuously differentiable solutions on $[a,b]$. But more importantly, the analysis becomes much more complex and we could not obtain a \emph{sharp} result, in the sense that, when $\alpha=2$ and $\beta=1$, our result would reduce to the one by Hartman and Wintner (cf. Corollary \ref{corHartman}). Nevertheless, our results generalize the well known de la Vallée Poussin inequality as well as the Fractional Lyapunov inequality.

We prove a series of lemmas before stating (and proving) our main result.

\begin{lem}
Let $x\in E_\beta:=\{f\in C^1(a,b]\cap C[a,b]:(D_a^\beta f)\in C[a,b]\}$ be a solution of \eqref{eq00}--\eqref{eq11}. Put $G(t)=g(t)(D_a^\beta x)(t)+f(t)x(t)$. Then,
\begin{multline}\label{porra1}
    (D_a^\beta x)(t)=\frac{1}{\Gamma(\alpha-\beta)}\left\{\int_a^t\left[\frac{(t-a)^{\alpha-\beta-1}(b-s)^{\alpha-1}}{(b-a)^{\alpha-1}}-(t-s)^{\alpha-\beta-1}\right]G(s)ds\right.\\ \left.+\int_t^b\frac{(t-a)^{\alpha-\beta-1}(b-s)^{\alpha-1}}{(b-a)^{\alpha-1}}G(s)ds\right\}
\end{multline}
\end{lem}

\begin{proof}
It is standard that $x\in E_\beta$ is a solution of \eqref{eq00}--\eqref{eq11} if and only if it satisfies the integral equation
$$x(t)=c(t-a)^{\alpha-1}-\frac{1}{\Gamma(\alpha)}\int_a^t(t-s)^{\alpha-1}G(s)ds.$$
The boundary condition at $t=b$ in \eqref{eq11} determines the constant $c$ and we get, 
$$x(t)=\frac{(t-a)^{\alpha-1}}{(b-a)^{\alpha-1}\Gamma(\alpha)}\int_a^b(b-s)^{\alpha-1}G(s)ds-\frac{1}{\Gamma(\alpha)}\int_a^t(t-s)^{\alpha-1}G(s)ds.$$
Finally, applying the Riemann--Liouville fractional derivative operator to both sides of the previous equality and having in mind that $(D_a^\beta (s-a)^{\alpha-1})(t)=\frac{\Gamma(\alpha)(t-a)^{\alpha-\beta-1}}{\Gamma(\alpha-\beta)}$ and Proposition \ref{Prop}, we get \eqref{porra1}.
\end{proof}

\begin{lem}\label{lem70}
Suppose that $\alpha-\beta-1\geq 0$. Define the function
$$f(t,s)=\frac{(t-a)^{\alpha-\beta-1}(b-s)^{\alpha-1}}{(b-a)^{\alpha-1}}-(t-s)^{\alpha-\beta-1},\ a\leq s\leq t\leq b.$$ Then,
$$|f(t,s)|\leq\max\left\{\frac{(s-a)^{\alpha-\beta-1}(b-s)^{\alpha-1}}{(b-a)^{\alpha-1}}:\alpha-\beta-1>0,(b-s)^{\alpha-\beta-1}-\frac{(b-s)^{\alpha-1}}{(b-a)^\beta}\right\}.$$
\end{lem}
\begin{proof}
We start by noticing that, if $\alpha-\beta-1=0$, then 
$$|f(t,s)|=\left|\frac{(b-s)^{\alpha-1}}{(b-a)^{\alpha-1}}-1\right|=1-\frac{(b-s)^{\alpha-1}}{(b-a)^{\alpha-1}}.$$
Suppose now that $\alpha-\beta-1>0$. Differentiating $f$ with respect to t and make some rearrangements gives
\begin{align*}
f_t(t,s)&=\frac{(\alpha-\beta-1)(t-a)^{\alpha-\beta-2}(b-s)^{\alpha-1}}{(b-a)^{\alpha-1}}-(\alpha-\beta-1)(t-s)^{\alpha-\beta-2},\quad a\leq s<t\leq b,\\
&=\frac{(\alpha-\beta-1)(t-a)^{\alpha-\beta-2}(b-s)^{\alpha-1}}{(b-a)^{\alpha-1}}\\
&\hspace{3cm}-(\alpha-\beta-1)\frac{(t-a)^{\alpha-\beta-2}}{(b-a)^{\alpha-\beta-2}}\left(b-\left(a+\frac{(s-a)(b-a)}{t-a}\right)\right)^{\alpha-\beta-2}\\
&=\frac{(\alpha-\beta-1)(t-a)^{\alpha-\beta-2}}{(b-a)^{\alpha-\beta-2}}\left[\frac{(b-s)^{\alpha-1}}{(b-a)^{\beta+1}}-\left(b-\left(a+\frac{(s-a)(b-a)}{t-a}\right)\right)^{\alpha-\beta-2}\right].
\end{align*}
Now, it is easy to see that 
$$a+\frac{(s-a)(b-a)}{t-a}\geq s\iff s\geq a,$$
hence
$$f_t(t,s)\leq \frac{(\alpha-\beta-1)(t-a)^{\alpha-\beta-2}}{(b-a)^{\alpha-\beta-2}}\left[\frac{(b-s)^{\alpha-1}}{(b-a)^{\beta+1}}-(b-s)^{\alpha-\beta-2}\right].$$
Observe now that 
$$\frac{(b-s)^{\alpha-1}}{(b-a)^{\beta+1}}-(b-s)^{\alpha-\beta-2}\leq 0\iff s\geq a,$$
which implies that $f_t(t,s)\leq 0$, i.e. $f$ is a decreasing function. Therefore,
$$|f(t,s)|\leq\max\{f(s,s),|f(b,s)|\},$$
from which the result follows.
\end{proof}
\begin{lem}\label{90}
Let $\alpha-\beta-1\geq 0$. Suppose that $G:[a,b]\rightarrow\mathbb{R}_0^+$. Define $F:[a,b]\rightarrow\mathbb{R}_0^+$ by
\begin{multline*}
F(t)=\\
\int_a^t\max\left\{\frac{(s-a)^{\alpha-\beta-1}(b-s)^{\alpha-1}}{(b-a)^{\alpha-1}}:\alpha-\beta-1>0,(b-s)^{\alpha-\beta-1}-\frac{(b-s)^{\alpha-1}}{(b-a)^\beta}\right\}G(s)ds\\
+\int_t^b\frac{(s-a)^{\alpha-\beta-1}(b-s)^{\alpha-1}}{(b-a)^{\alpha-1}}G(s)ds.
\end{multline*}
Then, 
\begin{multline*}
    F(t)\leq\\
\max\left\{\int_a^b\max\left\{\frac{(s-a)^{\alpha-\beta-1}(b-s)^{\alpha-1}}{(b-a)^{\alpha-1}}:\alpha-\beta-1>0,(b-s)^{\alpha-\beta-1}-\frac{(b-s)^{\alpha-1}}{(b-a)^\beta}\right\}G(s)ds\right.\\
\left.,\int_a^b\frac{(s-a)^{\alpha-\beta-1}(b-s)^{\alpha-1}}{(b-a)^{\alpha-1}}G(s)ds\right\}.
\end{multline*}
\end{lem}

\begin{proof}
We start by differentiating $F$ on $(a,b)$ to obtain
\begin{multline*}
F'(t)=
\left[\max\left\{\frac{(t-a)^{\alpha-\beta-1}(b-t)^{\alpha-1}}{(b-a)^{\alpha-1}}:\alpha-\beta-1>0,(b-t)^{\alpha-\beta-1}-\frac{(b-t)^{\alpha-1}}{(b-a)^\beta}\right\}\right.\\
\left.-\frac{(t-a)^{\alpha-\beta-1}(b-t)^{\alpha-1}}{(b-a)^{\alpha-1}}\right]G(t).
\end{multline*}
We claim that 
$$p(t)=\frac{(t-a)^{\alpha-\beta-1}(b-t)^{\alpha-1}}{(b-a)^{\alpha-1}},\quad \alpha-\beta-1>0,$$
and 
$$r(t)=(b-t)^{\alpha-\beta-1}-\frac{(b-t)^{\alpha-1}}{(b-a)^\beta},$$
coincide in exactly one point on $(a,b)$: indeed, it is easy to check that
$$p(t)=r(t)\iff\hat{p}(t)=\frac{(t-a)^{\alpha-\beta-1}(b-t)^\beta}{(b-a)^{\alpha-\beta-1}}=(b-a)^{\beta}-(b-t)^{\beta}=\hat{r}(t).$$
Differentiating twice the previous functions, it is not difficult to conclude that $\hat{p}(t)$ is concave while $\hat{r}(t)$ is convex. Noticing that $\hat{p}(a)=\hat{p}(b)=0$ and $\hat{r}(a)=0$, $\hat{r}(b)=(b-a)^\beta>0$ we conclude that $p$ and $r$ coincide in at most one point on $(a,b)$. However, it is not hard to see that $\hat{p}(\frac{a+b}{2})>\hat{r}(\frac{a+b}{2})$ and, since $\hat{p}(b)<\hat{r}(b)$, then continuity implies that there is a point $t^\star\in(\frac{a+b}{2},b)$ such that $p(t^\star)=r(t^\star)$, which concludes the proof of our claim.

Therefore, if 
$$\max\left\{\frac{(t-a)^{\alpha-\beta-1}(b-t)^{\alpha-1}}{(b-a)^{\alpha-1}}:\alpha-\beta-1>0,(b-t)^{\alpha-\beta-1}-\frac{(b-t)^{\alpha-1}}{(b-a)^\beta}\right\}=p(t),$$
then $F'(t)=0$ for all $t\in(a,t^\star)$, which implies that $F(t)=F(a)$ on that interval. On the other hand, if
$$\max\left\{\frac{(t-a)^{\alpha-\beta-1}(b-t)^{\alpha-1}}{(b-a)^{\alpha-1}}:\alpha-\beta-1>0,(b-t)^{\alpha-\beta-1}-\frac{(b-t)^{\alpha-1}}{(b-a)^\beta}\right\}=r(t),$$
then we define the function $X$ by
$$X(t)=r(t)-p(t)=(b-t)^{\alpha-1}\left[(b-t)^{-\beta}-\frac{(t-a)^{\alpha-\beta-1}}{(b-a)^{\alpha-1}}-(b-a)^{-\beta}\right].$$
Let $K(t)=(b-t)^{-\beta}-\frac{(t-a)^{\alpha-\beta-1}}{(b-a)^{\alpha-1}}-(b-a)^{-\beta}$. Then,
$$K'(t)=\beta(b-t)^{-\beta-1}-\frac{(\alpha-\beta-1)(t-a)^{\alpha-\beta-2}}{(b-a)^{\alpha-1}},$$
and
$$K''(t)=\beta(\beta+1)(b-t)^{-\beta-2}-\frac{(\alpha-\beta-1)(\alpha-\beta-2)(t-a)^{\alpha-\beta-3}}{(b-a)^{\alpha-1}}.$$
We see that $K''>0$ on $(a,b)$, which means that $K'$ is increasing. Now, if $\alpha-\beta-1=0$, then $K'>0$, hence $K$ is increasing. Since $K(a)=-\frac{1}{(b-a)^{\alpha-1}}$ and $\lim_{t\rightarrow b} K(t)=\infty$, then $X$ has a unique zero $t_\star\in(a,b)$ and $X(t)<0$ on $(a,t_\star)$, $X(t)>0$ on $(t_\star,b)$. Finally, suppose that $\alpha-\beta-1>0$. Since $\lim_{t\rightarrow a} K'(t)=-\infty$ and $\lim_{t\rightarrow b} K'(t)=\infty$ we conclude that $K'$ has a unique zero $\hat{t}\in(a,b)$. Moreover, we have that  $X(t)<0$ on $(a,\hat{t})$, $X(t)>0$ on $(\hat{t},b)$. Therefore, $F(t)\leq\max\{F(a),F(b)\}$ and the proof is done.
\end{proof}

It follows the main result of this section.

\begin{thm}\label{main1}
Fix $\alpha-\beta-1\geq 0$, with $1<\alpha\leq 2$ and $0<\beta\leq 1$. Suppose that $x\in E_\beta$ is a nontrivial solution of the BVP \eqref{eq00}--\eqref{eq11}. Then, the following inequality holds
\begin{multline*}
 \Gamma(\alpha-\beta)\leq\\
\max\left\{\int_a^b\max\left\{\frac{(s-a)^{\alpha-\beta-1}(b-s)^{\alpha-1}}{(b-a)^{\alpha-1}}:\alpha-\beta-1>0,(b-s)^{\alpha-\beta-1}-\frac{(b-s)^{\alpha-1}}{(b-a)^\beta}\right\}|g(s)|ds\right.\\
\left.,\int_a^b\frac{(s-a)^{\alpha-\beta-1}(b-s)^{\alpha-1}}{(b-a)^{\alpha-1}}|g(s)|ds\right\}\\
+\max\left\{\int_a^b\max\left\{\frac{(s-a)^{\alpha-\beta-1}(b-s)^{\alpha-1}}{(b-a)^{\alpha-1}}:\alpha-\beta-1>0,(b-s)^{\alpha-\beta-1}-\frac{(b-s)^{\alpha-1}}{(b-a)^\beta}\right\}\right.\\
\left.\cdot|f(s)|\frac{(s-a)^\beta}{\Gamma(\beta+1)}ds\right.\\
\left.,\int_a^b\frac{(s-a)^{\alpha-\beta-1}(b-s)^{\alpha-1}}{(b-a)^{\alpha-1}}|f(s)|\frac{(s-a)^\beta}{\Gamma(\beta+1)}ds\right\}.
\end{multline*}
\end{thm}
\begin{proof}
We have by \eqref{porra1} that
\begin{multline*}
    |(D_a^\beta x)(t)|\Gamma(\alpha-\beta)\leq\left\{\int_a^t\left|\frac{(t-a)^{\alpha-\beta-1}(b-s)^{\alpha-1}}{(b-a)^{\alpha-1}}-(t-s)^{\alpha-\beta-1}\right||G(s)|ds\right.\\ \left.+\int_t^b\frac{(t-a)^{\alpha-\beta-1}(b-s)^{\alpha-1}}{(b-a)^{\alpha-1}}|G(s)|ds\right\},
\end{multline*}
where $G(t)=g(t)(D_a^\beta x)(t)+f(t)x(t)$. Now, let $\mu=\max_{t\in[a,b]}|(D_a^\beta x)(t)|>0$. Using Theorem \ref{thm23}, we get
$$|G(t)|\leq |g(t)|\mu+|f(t)|\frac{(t-a)^\beta}{\Gamma(\beta+1)}\mu.$$
Inserting this inequality in the previous one, we achieve
\begin{multline*}
    \Gamma(\alpha-\beta)\leq\left\{\int_a^t\left|\frac{(t-a)^{\alpha-\beta-1}(b-s)^{\alpha-1}}{(b-a)^{\alpha-1}}-(t-s)^{\alpha-\beta-1}\right|\left[|g(s)|+|f(s)|\frac{(s-a)^\beta}{\Gamma(\beta+1)}\right]ds\right.\\ \left.+\int_t^b\frac{(s-a)^{\alpha-\beta-1}(b-s)^{\alpha-1}}{(b-a)^{\alpha-1}}\left[|g(s)|+|f(s)|\frac{(s-a)^\beta}{\Gamma(\beta+1)}\right]ds\right\}\\
    =\left\{\int_a^t\left|\frac{(t-a)^{\alpha-\beta-1}(b-s)^{\alpha-1}}{(b-a)^{\alpha-1}}-(t-s)^{\alpha-\beta-1}\right||g(s)|ds\right.\\
    \left.+\int_t^b\frac{(s-a)^{\alpha-\beta-1}(b-s)^{\alpha-1}}{(b-a)^{\alpha-1}}|g(s)|ds\right.\\ \left.+\int_a^t\left|\frac{(t-a)^{\alpha-\beta-1}(b-s)^{\alpha-1}}{(b-a)^{\alpha-1}}-(t-s)^{\alpha-\beta-1}\right||f(s)|\frac{(s-a)^\beta}{\Gamma(\beta+1)}ds\right.\\
    \left.+\int_t^b\frac{(s-a)^{\alpha-\beta-1}(b-s)^{\alpha-1}}{(b-a)^{\alpha-1}}|f(s)|\frac{(s-a)^\beta}{\Gamma(\beta+1)}ds\right\}.
\end{multline*}
An application of Lemma \ref{lem70} and afterwards of Lemma \ref{90} finally yields
\begin{multline*}
 \Gamma(\alpha-\beta)\leq\\
\max\left\{\int_a^b\max\left\{\frac{(s-a)^{\alpha-\beta-1}(b-s)^{\alpha-1}}{(b-a)^{\alpha-1}}:\alpha-\beta-1>0,(b-s)^{\alpha-\beta-1}-\frac{(b-s)^{\alpha-1}}{(b-a)^\beta}\right\}|g(s)|ds\right.\\
\left.,\int_a^b\frac{(s-a)^{\alpha-\beta-1}(b-s)^{\alpha-1}}{(b-a)^{\alpha-1}}|g(s)|ds\right\}\\
+\max\left\{\int_a^b\max\left\{\frac{(s-a)^{\alpha-\beta-1}(b-s)^{\alpha-1}}{(b-a)^{\alpha-1}}:\alpha-\beta-1>0,(b-s)^{\alpha-\beta-1}-\frac{(b-s)^{\alpha-1}}{(b-a)^\beta}\right\}\right.\\
\left.\cdot|f(s)|\frac{(s-a)^\beta}{\Gamma(\beta+1)}ds\right.\\
\left.,\int_a^b\frac{(s-a)^{\alpha-\beta-1}(b-s)^{\alpha-1}}{(b-a)^{\alpha-1}}|f(s)|\frac{(s-a)^\beta}{\Gamma(\beta+1)}ds\right\}.
\end{multline*}
The proof is done.
\end{proof}
The following result shows that Theorem \ref{main1} is a generalization of the de la Vallée Poussin inequality.

\begin{cor}
Theorem \ref{VP} is a consequence of Theorem \ref{main1}.
\end{cor}
\begin{proof}
Put $\alpha=2$ and $\beta=1$ in Theorem \ref{main1}. Then,
\begin{multline*}
 1\leq
\max\left\{\int_a^b\frac{s-a}{b-a}|g(s)|ds,\int_a^b\frac{b-s}{b-a}|g(s)|ds\right\}\\
+\max\left\{\int_a^b\frac{(s-a)^2}{b-a}|f(s)|ds,\int_a^b\frac{(b-s)(s-a)}{b-a}|f(s)|ds\right\}\\
<(b-a)M_1+M_2\max\left\{\frac{(b-a)^2}{3},\frac{(b-a)^2}{2}\right\}=M_1(b-a)+M_2\frac{(b-a)^2}{2},
\end{multline*}
which concludes the proof.
\end{proof}
Another consequence of Theorem \ref{main1} is the fractional Lyapunov inequality, that was firstly established by the author in \cite{Ferreira0}.

\begin{cor}
If the following fractional boundary value problem
\begin{align*}
    (D_a^\alpha x)&+f(t)x=0,\quad t\in(a,b),\ 1<\alpha\leq 2,\\
    x(a)=&0=x(b),
\end{align*}
where $q\in C[a,b]$ has a nontrivial solution, then
$$\int_a^b|f(s)|ds>\Gamma(\alpha)\left(\frac{4}{b-a}\right)^{\alpha-1}.$$
\end{cor}
\begin{proof}
In Theorem \ref{main1} we let $g=0$ on $[a,b]$. Then, we may take $\beta=0$ and we have that 
$$\Gamma(\alpha)\leq\int_a^b\frac{(s-a)^{\alpha-1}(b-s)^{\alpha-1}}{(b-a)^{\alpha-1}}|f(s)|ds$$
Now, note that $f$ cannot be zero on the entire interval $[a,b]$, otherwise, $x$ would be the trivial solution. Therefore, by using \cite[Lemma 2.2]{Ferreira0}, we get
$$\int_a^b\frac{(s-a)^{\alpha-1}(b-s)^{\alpha-1}}{(b-a)^{\alpha-1}}|f(s)|ds<\left(\frac{b-a}{4}\right)^{\alpha-1}\int_a^b|f(s)|ds,$$
from which the result follows.
\end{proof}

We end this work establishing a result analogous to Theorem \ref{thmzeros}.

\begin{thm}
Let $1<\alpha\leq 2$ and $0<\beta\leq 1$ be such that $\alpha-\beta-1\geq 0$. Then, the Mittag--Leffler function $$E_{\alpha-\beta,\alpha}(x)=\sum_{k=0}^\infty\frac{x^k}{\Gamma(k(\alpha-\beta)+\alpha)},$$
has no real zeros for $x\in(-\nu,0)$, where
$$\nu=\frac{\Gamma(\alpha-\beta)}{\max\left\{\int_0^1\Delta(s)ds,B(\alpha-\beta,\alpha)\right\}},$$
with $\Delta(s)=\max\left\{s^{\alpha-\beta-1}(1-s)^{\alpha-1}:\alpha-\beta-1>0,(1-s)^{\alpha-\beta-1}-(1-s)^{\alpha-1}\right\}$ and $B(x,y)$ being the Beta function. 
\end{thm}

\begin{proof}
Consider $a=0$ and $b=1$. Let $f=0$ in \eqref{eq00} and suppose that $x$ is a nontrivial solution of the following BVP
\begin{align*}
    D_0^\alpha x(t)+\lambda(D_0^\beta x)(t)&=0,\quad t\in(0,1),\ \lambda\in\mathbb{R},\\
    x(0)=&0=x(1).
\end{align*}
By \cite[Corollary 5.3]{Kilbas} we know that $\lambda$ must satisfy $E_{\alpha-\beta,\alpha}(-\lambda)=0$. It is clear that, if such $\lambda$ exist, it must be positive. Using Theorem \ref{main1}, we obtain
\begin{multline*}
\Gamma(\alpha-\beta)\leq\\
\lambda\max\left\{\int_0^1\max\left\{s^{\alpha-\beta-1}(1-s)^{\alpha-1}:\alpha-\beta-1>0,(1-s)^{\alpha-\beta-1}-(1-s)^{\alpha-1}\right\}ds\right.\\
\left.,\int_0^1s^{\alpha-\beta-1}(1-s)^{\alpha-1}ds\right\}.
\end{multline*}
Noting that $\int_0^1s^{\alpha-\beta-1}(1-s)^{\alpha-1}ds=B(\alpha-\beta,\alpha)$, where $B(x,y)$ is the Beta function, we finally achieve the result we wanted to prove.
\end{proof}

\section*{Acknowledgments}

Rui Ferreira was supported by the ``Funda\c{c}\~ao para a Ci\^encia e a Tecnologia (FCT)" through the program ``Investigador FCT" with reference IF/01345/2014.

%% References
%%
%% Following citation commands can be used in the body text:
%% Usage of \cite is as follows:
%%   \cite{key}         ==>>  [#]
%%   \cite[chap. 2]{key} ==>> [#, chap. 2]
%%

%% References with bibTeX database:

\end{document}